\documentclass{amsart}
\usepackage{latexsym,enumerate}
\usepackage{amsmath,amsthm,amsopn,amstext,amscd,amsfonts,amssymb}
\usepackage[ansinew]{inputenc}
\usepackage{verbatim}
\usepackage{graphicx}
\usepackage{pstricks}
\usepackage{wasysym}
\usepackage[all]{xy}
\usepackage{epsfig} 
\usepackage{graphicx,psfrag}
\usepackage{subfigure}
\usepackage{pstricks,pst-node}
\usepackage{multicol}
\usepackage{color}
\usepackage{colortbl}
\newcommand{\NN}{\mathbb{N}}

\newcommand{\RR}{\mathbb{R}}
\newcommand{\ZZ}{\mathbb{Z}}

\newtheorem{theorem}{Theorem}[section]
\newtheorem{lemma}[theorem]{Lemma}
\newtheorem{proposition}[theorem]{Proposition}
\newtheorem{corollary}[theorem]{Corollary}
\newtheorem{definition}[theorem]{Definition}
\newtheorem{example}[theorem]{Example}
\newtheorem{remark}[theorem]{Remark}
\DeclareMathOperator{\diam}{diam}

\newcommand{\spb}[1]{\smallskip}
\newcommand{\mpb}[1]{\medskip}
\newcommand{\bpb}[1]{\bigskip}

\begin{document}
	\DeclareGraphicsExtensions{.jpg,.pdf,.mps,.png}
	\title{Vertex separators, chordality and virtually free groups}
	\author[Samuel G. Corregidor]{Samuel G. Corregidor}
	\address{ Facultad CC. Matem\'aticas, Universidad Complutense de Madrid,
		Plaza de Ciencias, 3. 28040 Madrid, Spain}
	\email{samuguti@ucm.es}
	\author[\'{A}lvaro Mart\'{\i}nez-P\'erez]{\'{A}lvaro Mart\'{\i}nez-P\'erez$^{(1)}$}
	\address{ Facultad CC. Sociales, Universidad de Castilla-La Mancha,
		Avda. Real F\'abrica de Seda, s/n. 45600 Talavera de la Reina, Toledo, Spain}
	\email{alvaro.martinezperez@uclm.es}
	\thanks{$^{(1)}$ Supported in part by a grant
		from Ministerio de Ciencia, Innovaci\'on y Universidades (PGC2018-098321-B-I00), Spain.
	}
	\date{\today}
	\begin{abstract} In this paper we consider some results obtained for graphs using minimal vertex separators and generalized chordality and translate them to the context of Geometric Group Theory.	
		Using these new tools, we are able to give two new characterizations for a group to be virtually free. Furthermore, we  prove that the Baumslag-Solitar group $BS(1,n)$ is $k$-chordal for some $k$ if and only if $|n|<3$ and we give an application of generalized chordality to the study of the word problem.
	\end{abstract}
	\maketitle{}
	{\it Keywords: vertex separator, chordality, hyperbolicity, quasi-isometry, tree, virtually free.} 
	{\it 2020 AMS Subject Classification Numbers. Primary: 20F65. Secondary: 20F10, 05C25} 
	
	\section{Introduction}


	In \cite{M2} the second author studied some relations between vertex separator sets, certain
	chordality properties that generalize being chordal and conditions for a graph to be quasi-isometric to a tree. Some of these ideas can be easily translated to the language of Geometric Group Theory. Moreover, some of the definitions can be re-written in new terms given a group presentation.
	
	Thus, we introduce herein new tools in the context of Geometric Group Theory which can be used from either a more geometric or more algebraic approach. 
	
	Let $\Gamma$ be a connected graph, we will understand it as a geodesical metric space. A subset $P\subseteq V(\Gamma)$ is an \emph{$(a,b)$-separator} in $\Gamma$ if the vertices $a$ and $b$ are in different connected components of $\Gamma\smallsetminus P$.  
	
	A graph $\Gamma$ satisfies the \emph{bottleneck property} (BP) if there exists some constant $\Delta > 0$ so that given any two distinct points $x,y \in V(\Gamma) $ and a midpoint $c \in \Gamma $ such that $d_\Gamma(x,c) = d_\Gamma(y,c) = \frac{1}{2} d_\Gamma(x,y)$, then every $xy$-path intersects $N_\Delta (c)$. This is an equivalent definition to the original which was defined in \cite{Man}, see Proposition 2 in \cite{M2}.
	
	Given $P \subseteq V(\Gamma)$, we denote:
	
	$$ \diam (P) = \sup\{ d_\Gamma (p,q) \,|\, p,q \in P \} $$
	
	Then we have the following result for graphs.
	
	\begin{theorem}[Theorem 7, \cite{M2}]\label{t:separator}Let $\Gamma$ be a uniform graph. If every minimal vertex separator in $\Gamma$ has diameter at most $m$, then $\Gamma$ satisfies (BP) (i.e., $\Gamma$ is quasi-isometric to a tree).
	\end{theorem}
	
	This is translated in Section \ref{sep} to Geometric Group Theory. As a result, we obtain in Theorem \ref{t:char_vfree} a new characterization for a group to be virtually free which is somehow a re-writting of (BP).
	
	\smallskip
	
	By a \emph{cycle} in a graph we mean a simple closed curve $\gamma$, i.e., a path with different vertices, except for the last one, wich is equal to the first vertex. Without loss of generality we will suppose that every cycle has length at least 3. A \emph{shortcut} in $ \gamma $ is a path $\sigma$ joining two vertices $p,q$ in $ \gamma $ such that $L(\sigma) < d_\gamma (p,q) $, where $ L(\sigma) $ denotes the length of the path $\sigma $ and $d_\gamma$ denotes the length metric on $\gamma$. A shortcut $\sigma$ in $\gamma$ between $p,q$ is \emph{strict} if $\sigma \cap \gamma = \{ p,q \}$. In this case, we say that $p,q$ are \emph{shortcut vertices} in $ \gamma $ \emph{associated} with $\sigma$. 
	
	A graph $ \Gamma $ is \emph{$(k,m)$-chordal} if for every cycle $\gamma$ in $\Gamma$ with length $L(\gamma) \geqslant k $, there exists a shortcut $\sigma $ such that $L(\sigma ) \leqslant m $. The graph $\Gamma$ is \emph{$k$-chordal} if we make $m=\infty $.

	A graph $ \Gamma $ is \emph{$\varepsilon$-densely $(k,m)$-chordal} if for every cycle $\gamma$ in $\Gamma$ with length $L(\gamma) \geqslant k $, there exist strict shortcuts $\sigma _1 , \ldots , \sigma _r $ with $L(\sigma _i) \leqslant m \, \forall\, i, $ and such that their associated shortcut vertices define an $\varepsilon$-dense subset in $(\gamma,d_\gamma)$. The graph $\Gamma$ is \emph{$\varepsilon$-densely $k$-chordal} if we make $m=\infty $. 
	
	\begin{remark}\label{r:natural_ekm_chordal}If $ \Gamma $ is $\varepsilon$-densely $(k,m)$-chordal, then $ \Gamma $ is $\left\lceil  \varepsilon  \right\rceil$-densely $(\left\lceil  k  \right\rceil,\left\lceil  m \right\rceil)$-chordal. Therefore, we can suppose always that $\varepsilon,k,m \in\NN $.
	\end{remark}
	
	We proved that this property characterizes being quasi-isometric to a tree for graphs:
	
	\begin{theorem}[Corollary 3, \cite{M2}]\label{t:qi_tree}A graph $\Gamma$ is quasi-isometric to a tree if and only if it is $\varepsilon$-densely $(k,m)$-chordal for some $\varepsilon,k,m $.
	\end{theorem}

	A group $G$ is \emph{virtually free} if it has a free subgroup of finite index. There are many characterizations of being virtually free in the literature. Let us recall here some of them.

	
	
	
	\begin{theorem}\label{T:charac} Let $\Gamma$ be the Cayley graph of a finitely generated group $G$, then the following are equivalent:
		\begin{enumerate}
			\item[(1)] $G$ is virtually free.
			\item[(2)] The language of all words on $S^{\pm 1} $ representing the identity is context-free, where $S$ is a finite generating set. (Muller-Schupp's Theorem \cite{D,MS})
			\item[(3)] $G$ is context-free. \cite{D,MS}
			\item[(4)] There exist $k\geqslant 0$ such that every closed path in $\Gamma$ is $k$-triangulable. \cite{MS}
			\item[(5)] The ends of $\Gamma$ have finite diameter. \cite{W}
			\item[(6)] $\Gamma$ admits a uniform spanning tree. \cite{W}
			\item[(7)] $\Gamma$ is quasi-isometric to a Cayley graph of a free group. \cite{D,GH}
			\item[(8)] $\Gamma$ has finite strong tree-width. \cite{KL}
			\item[(9)] $G$ is finitely presentable and asdim$\Gamma$ = 1, where asdim$\Gamma$ is the asymptotic dimension. \cite{FW,Gen,JS}
			\item[(10)] There exists a generating set of $G$ and $k\geqslant 0$ such that every $k$-locally geodesic in $G$ is a geodesic. \cite{GHHR}
			\item[(11)] There exists a generating set of $G$ such that $G$ admits a (regular) broomlike combing with respect to the generating set. \cite{BG}
			\item[(12)] $G$ acts on a tree with finite stabilizers.\cite{A,DD,KPS}
			\item[(13)] $G$ is the fundamental group of a finite graph of finite groups. \cite{A,DD,KPS}
			\item[(14)] $\Gamma$ is quasi-isometric to a tree. \cite{D,S}
			\item[(15)] $\Gamma$ satisfies (BP). \cite{Man}
			\item[(16)] $G$ admits a finite presentation by some geodesic rewriting system. \cite{GHHR}
			\item[(17)] $G$ is the universal group of a finite pregroup. \cite{R}
			\item[(18)] The monadic second-order theory of $\Gamma$ is decidable. \cite{KL,MS2}
			\item[(19)] $G$ admits a Stallings section. \cite{SSV}
			\item[(20)] $\overline{\Gamma}$ is polygon hyperbolic. \cite{AS}
			\item[(21)] For any finite generating set, the corresponding Cayley graph is minor excluded. \cite{K}
		\end{enumerate}
	\end{theorem}
	
	For some alternative proofs of some of these see also \cite{A,DW}.
	
	\smallskip
	
	In Section \ref{Chor} we translate the definition of being $\varepsilon$-densely $(k,m)$-chordal to the language of Group Theory, see Theorem \ref{t:ikm_chordal}. In those terms, we give a new characterization of being virtually free in Corollary \ref{c:virtually_free} which states the following:
	
	\begin{theorem}	A finitely generated group $G$ is virtually free if and only if $G$ is $(i_0,k,m)$-chordal for some $i_0,k,m$.
	\end{theorem}
	
	
	
	
	
	
	

	\smallskip
	
	In Section \ref{BS} we analyze the example of Baumslag-Solitar groups and prove that $BS(1,n)$ is $k$-chordal for some $k$ if and only if $|n|<3.$

	\smallskip

	The word problem in the group presentation $G = \left\langle S | R \right\rangle $ consists in finding an 
	algorithm that takes as its input a word $w$ over $S$ and answers if $w$ represents the trivial element in $G$ or not. For many groups this problem is undecidable. A group has solvable word problem if such an algorithm exists. For details about the word problem, see \cite{DK,L}. 

	In Section \ref{word} we give an application of generalized chordality to the study of the word problem for groups. In particular we obtain Theorem \ref{t:word_problem} which states the following:
	
	\begin{theorem} Let $ G = \left\langle S | R \right\rangle $ be a finite group presentation. If $G$ is $k$-chordal, then $\left\langle S | R \right\rangle$ admits a recursive isoperimetric function, i.e., $\left\langle S | R \right\rangle$ has a solvable word problem.
	\end{theorem}
	
	Moreover, we obtain that if $G$ is $k$-chordal, then there is an isoperimetric function $f$ such that $f(n)=c2^{n-k}$ for some constant $c$.
	
	Some groups are known to have solvable word problem: free groups, see \cite{DK}; finitely presented residually finite groups, see \cite{DK}; hyperbolic groups, \cite{G}, among others. In particular, it is well known that the free abelian group $ \left\langle a,b\,|\, [a,b] \right\rangle  $ has solvable word problem, see \cite{L}. We prove again this fact to illustrate an application of Theorem \ref{t:word_problem} in search of groups with solvable word problem.

	\section{Background}\label{S:back}
	
	Let $(X,d_1)$, $(Y,d_2)$ be metric spaces, $f:\, X \longrightarrow Y$ and $a,b > 0$. The function $f$ is an \emph{$(a,b)$-quasi-isometric embedding} if $$ \frac{1}{a} d_1 (x,y) - b \leqslant d_2 (f(x),f(y)) \leqslant a d_1 (x,y) + b,\,\forall\, x,y\in X. $$ 
	Two functions $f,g:\, X \longrightarrow Y$ are said to be at \emph{finite distance} if there exists $\varepsilon > 0$ such that $ d_2 (f(x) , g(x)) < \varepsilon $ for all $x\in X$. The function $f$ is a \emph{quasi-isometry} if $f$ is a quasi-isometric embedding (there exist $a,b > 0$ such that $f$ is an $(a,b)$-quasi-isometric embedding) and there exists a quasi-isometric embedding  $g:\, Y \longrightarrow X$ such that $f\circ g $ and $Id_Y$ are at finite distance and $g\circ f $ and $Id_X$ are also at finite distance. Let $\varepsilon > 0$ be a constant. A subset $D\subseteq X$ is said to be \emph{$\varepsilon$-dense} if given a point $x\in X$ there exists $p\in P$ such that $d_1( x,p ) < \varepsilon$. It is well known that $f$ is a quasi-isometry if and only if $f$ is a quasi-isometric embedding and $f(X)$ is $\varepsilon$-dense for some $\varepsilon > 0$, see \cite{L} for details. It is said that two metric spaces are \emph{quasi-isometric} if there is a quasi-isometry between both of them.
	
	\smallskip
	
	Given a metric space $(X,d)$, a \emph{geodesic} from $x$ to $y$ is an isometry  $\gamma :\, [0,L] \longrightarrow X$ such that $\gamma (0) = x$ and $\gamma (L) = y$. Abusing the notation, we will identify the geodesic $\gamma$ with its image and refer to this image as a geodesic. A metric space $(X,d)$ is \emph{geodesic} when there is a geodesic from $x$ to $y$ for every pair of points $x,y\in X$. It is important to note that geodesics do not have to be unique, for convenience, $[xy]$ will denote any such geodesic.
	
	\smallskip
	
	There are several definitions of Gromov $\delta$-hyperbolic space which are equivalent although the constant $\delta$ may appear multiplied by some constant, see \cite{BS} for details. Since we will use geodesic spaces, it is natural to use the Rips condition on the triangles. Given three points $x,y,z$ in a geodesic space $(X,d)$, a geodesic $\gamma_1$ from $x$ to $y$, a geodesic $\gamma_2$ from $y$ to $z$ and a geodesic $\gamma_3$ from $z$ to $x$, the union of these geodesics is a \emph{geodesic triangle} and will be denoted by $T=\{ x,y,z \}$ or $T=\{ \gamma_1 , \gamma_2, \gamma_3 \}$ when we want to emphasize geodesics. $T$ is \emph{$\delta$-thin} if any side (any geodesic $\gamma_i$) of $T$ is containded in the $\delta$-neighborhood of the union of the two other sides. A geodesic metric space $(X,d)$ is \emph{$\delta$-hyperbolic} if every geodesic triangle in $(X,d)$ is $\delta$-thin. A geodesic metric space $(X,d)$ is \emph{hyperbolic} if $(X,d)$ is $\delta$-hyperbolic for some $\delta \geqslant 0$. 
	
	\smallskip
	
	A connected graph $\Gamma$ can be interpreted as a metric space with a length metric $d_\Gamma$ by considering every edge $e\in E(\Gamma)$ as isometric to the interval $[0,1]$. Therefore, the distance between any pair of points $x,y\in \Gamma$ will be the length of the shortest path in $\Gamma$ joining $x$ and $y$. In this paper we will assume that every graph is locally finite and connected. Thus, every graph is a geodesic metric space.
	
	For any vertex $v\in V(\Gamma) $ and any constant $\varepsilon > 0$, let us denote:
	
	$$ S_\varepsilon (v) := \{ w\in V(\Gamma)\, |\, d(v,w) = \varepsilon \}, $$
	
	$$ B_\varepsilon (v) := \{ w\in V(\Gamma)\, |\, d(v,w) < \varepsilon \}, $$
	
	$$ N_\varepsilon (v) := \{ w\in V(\Gamma)\, |\, d(v,w) \leqslant \varepsilon \}. $$
	
	\smallskip
	
	Let $G$ be a group generated by a subset $S\subseteq G \smallsetminus \{e\}$ and let $S^{\pm 1} = S\cup S^{-1} $. $S$ is a generating set if every element of $G$ can be written as a finite product of elements in $S^{\pm 1}$. $G$ is finitely generated if there is a finite generator set. 
	The \emph{Cayley graph} of $G$ with respect to $S$ is a graph $\Gamma = Cay(G,S)$ whose set of vertices is $G$ and whose set of edges is $ \{ \{ g,gs \}\,|\, g\in G, \, s\in S^{\pm 1} \} $. Given two finite generating sets $S$ and $S'$ of $G$, then $Cay(G,S)$ and $Cay(G,S')$ are quasi-isometric (via $id_G$), see \cite{L} for details. A finitely generated group $G$ is \emph{quasi-isometric} to a metric space $(X,d)$ if there is a quasi-isometry between $\Gamma$ and $(X,d)$. This is well defined because it does not depend on the chosen generating sets. Hyperbolicity is a well-known invariant under quasi-isometries. Therefore, it can be said that a finitely generated group $G$ is \emph{hyperbolic} if some Cayley graph $\Gamma$ of $G$ is hyperbolic.
	
	\section{Vertex separators and virtually free groups}\label{sep}

	\begin{definition}\label{d:vertex_separator}Let $G$ be a group generated by $S$. We say $P\subseteq G$ is an $(a,b)$-separator with respect to $S$ if for all $s_1, \ldots , s_n \in S^{\pm 1} $ such that $ b = as_1 \cdots s_n $, then there exists $ 0 < i < n $ such that $ as_1 \cdots s_i \in P$. A vertex separator with respect to $S$ is an $(a,b)$-separator with respect to $S$ for some $a,b\in G$. A minimal vertex $(a,b)$-separator with respect to $S$ is an $(a,b)$-separator with respect to $S$ for some $a,b\in G$ such that no proper subset is an $(a,b)$-separator with respect to $S$.
	\end{definition}
	
	If $\Gamma$ is the Cayley graph of $G$ with respect to $S$ and $P\subset G=V(\Gamma)$, it is immediate that $P$ is an $(a,b)$-separator in $\Gamma$ if and only if $P$ is an $(a,b)$-separator with respect to $S$.
	
	\begin{remark}\label{r:vertex_separator}Let $G$ be a group generated by $S$. The equation $ b = as_1 \cdots s_n $ can be read as $ a^{-1}b =  s_1 \cdots s_n $. Therefore, $P$ is an $(a,b)$-separator with respect to $S$ if and only if $a^{-1} P $ is an $(e ,a^{-1}b)$-separator with respect to $S$.
	\end{remark}
	
	Let $G$ be a group generated by $S$ and $A\subseteq G$. We denote the diameter of $A$ (with respect to $S$) as $ \diam (A) = \sup\{ d_\Gamma (x,y) \,|\, x,y\in A \} $. The set of vertex separators between $e$ and the rest of elements of $G$ will be denoted as follows. 
	$$ Sep(G)=\{P \, | \, P \text{ is a minimal $(e,x)$-separator with respect to $S$ for some } x\in G\smallsetminus\{e\}\}.$$
	We denote $ \mathcal{D}(G) = \sup\{ \diam (P) \,|\, P\in Sep(G) \} $.
	
	
	\begin{theorem}\label{t:separator_groups}Let $G$ be a group generated by $S$. If $ \mathcal{D}(G) < \infty $, then $G$ is virtually free.
	\end{theorem}
	
	\begin{proof}Suppose $ \mathcal{D}(G) < \infty $. Let $P$ be a minimal vertex separator between $a$ and $b$. By Remark \ref{r:vertex_separator}, $ a^{-1} P \in Sep(G) $ and therefore $\diam (P) = \diam (a^{-1} P) < \mathcal{D}(G) < \infty $. By Theorem \ref{t:separator}, $Cay(G,S)$ satisfies (BP), i.e., $G$ is virtually free.
	\end{proof}
	
	The following example shows that the converse is not true.
	
	\begin{example}\label{e:converse_counter}Let $G = \ZZ\times\ZZ _4$ be a group generated by $ S = \{ (1,0),(0,1) \} $. Obviously, $G$ is virtually free. However, $P= \ZZ\times\{1 \} \cup \ZZ\times\{3 \} $ is a minimal separator with infinite diameter, see Figure \ref{Fig:converse_counter}.
	\end{example}
	
	\begin{figure}[h]
		\centering
		\includegraphics[scale=0.6]{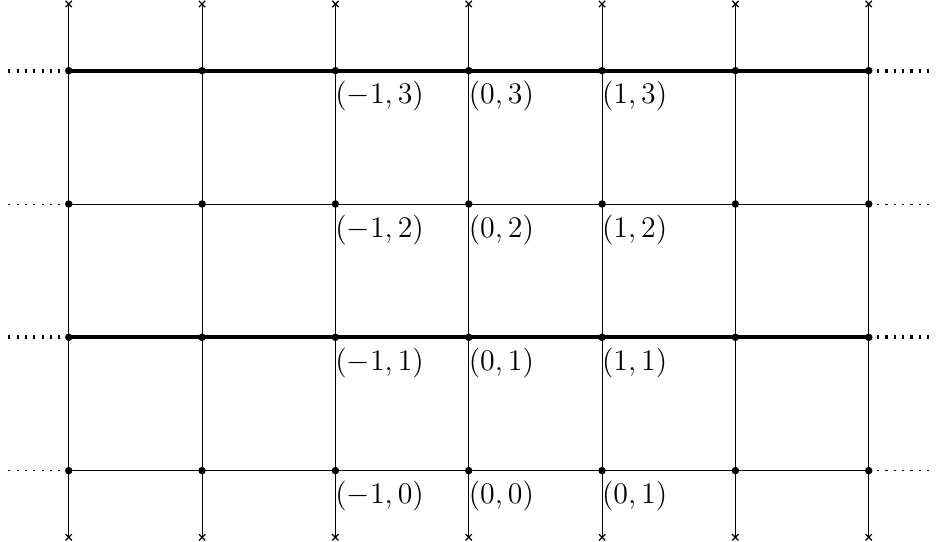}
		\caption{The Cayley graph of $G$ with respect to $S$ (the cross in vertical lines are glued). In bold the minimal vertex separator $P$.}
		\label{Fig:converse_counter}
	\end{figure}
	
	However, Theorem \ref{t:separator_groups} can be improved to obtain a characterization for being virtually free in terms of minimal vertex separators. The result, Theorem \ref{t:char_vfree} below, is quite a straightforward re-writing of Manning's Bottleneck Property.
	
	\begin{definition}\label{d:separator_BP}Let $G$ be a group generated by $S$. We
		say that $\mathcal{F} = \{ P_x \, : \, x\in G\smallsetminus N_1(e) \}$ is a \textit{$k$-separating family} with respect to $S$ if for every $x\in G\smallsetminus N_1(e)$:
		\begin{enumerate}
			\item $P_x$ is a minimal vertex $(e,x)$-separator with respect to $S$.
			\item There is $v\in P_x$ such that $d(v,c) \leqslant k$, where $c$ is some midpoint of a geodesic $[ex]$.
		\end{enumerate}
	\end{definition}
	
	From Lemma 1 in \cite{M2} it follows that for every finitely generated group there exists a $\frac{1}{2}$-separating family and, therefore, a $k$-separating family for every $k\geq \frac12$.
	
	\smallskip
	
	Given a $k$-separating family $\mathcal{F}$, we denote $ \mathcal{D}(\mathcal{F}) = \sup\{ \diam (P) \,|\, P\in \mathcal{F} \}  $ and 
	$$ \hat{\mathcal{D}}(G) = \inf \{ \mathcal{D}(\mathcal{F}) \,|\, \mathcal{F} \text{ is a } k\text{-separating family for some } k > 0 \}.$$


	\begin{theorem}\label{t:char_vfree}Let $G$ be a group generated by $S$. The group $G$ is virtually free  if and only if $ \hat{\mathcal{D}}(G) < \infty$.
	\end{theorem}
	
	\begin{proof}Suppose that $ \hat{\mathcal{D}}(G) < \infty $. Then there is a $k$-separating family, $\mathcal{F}$, such that $ D = \sup\{ \diam (P_x) \,|\, P_x\in \mathcal{F} \} < \infty $. We claim that $\Gamma = Cay(G,S)$ satisfies $(BP)$.
		
		Consider an element $x \in G\smallsetminus N_1(e) $ and a midpoint $m \in \Gamma = Cay(G,S) $ such that $d_\Gamma(x,m) = d_\Gamma(e,m) = \frac{1}{2} d_\Gamma(e,x)$.
		
		If $d_\Gamma (e,x) = 2$, then (BP) is trivial. Suppose $d_\Gamma (e,x) > 2$. Then, there is some vertex $v\in P_x$ such that $d_\Gamma(v,m) \leqslant k$. Thus, since every $ex$-path contains a vertex in $P_x$ and  $\diam (P_x) \leqslant D $, then every $ex$-path intersects $ N_D (v) \subset N_{D+k} (m)$. Hence, (BP) is satisfied and therefore $G$ is virtually free.
		
		Suppose that $G$ is virtually free. Then, (BP) is satisfied with some constant $\Delta > 0$. Thus, for every  $x\in G\smallsetminus N_1(e)$ and every midpoint $m \in \Gamma $ such that $d_\Gamma(x,m) = d_\Gamma(e,m) = \frac{1}{2} d_\Gamma(e,x)$,  every $ex$-path intersects $N_\Delta (m)$. Pick a vertex $v_x\in N_\frac{1}{2} (m) $, then $ V(\Gamma)\cap N_{\Delta+\frac12} (v_x) $ contains a minimal $(e,x)$-separator $P_x$. The family $\mathcal{F} = \{ P_x \, : \, x\in G\smallsetminus N_1(e) \}$ is a $\frac{1}{2}$-separating family with respect to $S$ such that $ \mathcal{D}(\mathcal{F}) \leqslant 2\Delta +1  $. Thus, $ \hat{\mathcal{D}}(G) \leq  \mathcal{D}(\mathcal{F}) < \infty $.
	\end{proof}

	\section{Chordality on groups}\label{Chor}

	\begin{definition}\label{d:reduced_relation}Let $G$ be a group generated by $S$, $n>2$ and $ s_1,\ldots , s_n \in S^{\pm 1} $. We say that $ s_1 \cdots s_n = e $ is a simple relation in $ S^{\pm 1} $ when $$ s_p s_{p+1} \cdots s_{q-1} s_q = e \text{ if and only if }(p,q) = (1,n)$$
	\end{definition}
	
	\begin{remark}\label{r:reduced_relation}Let $\gamma$ be a cycle in $\Gamma$ defined by $ g_0,\ldots,g_n = g_0 $, then $ u_1 \cdots u_n = e $ is a simple relation in $S^{\pm 1}$ where $ u_k = g_{k-1}^{-1} g_k ,\, k\in\{ 1,\ldots, n \}$. On the other hand, if $ s_1 \cdots s_n = e $ is a simple relation in $S^{\pm 1}$, then $ g_0 = e,\, g_k = s_1 \cdots s_k,\, k\in\{ 1,\ldots,n \} $, defines a cycle in $\Gamma$.
	\end{remark}
	
	\begin{definition}\label{d:k,m-chordal_group}Let $G$ be a group generated by $S$ and $k,m\in\NN $. We say that $G$ is $(k,m)$-chordal with respect to $S$ if for every simple relation $ s_1 \cdots s_n = e $ in $S^{\pm 1}$ with $n \geqslant k$, then there exist $s_1 ' , \ldots , s_r ' \in S^{\pm 1}$ such that $ s_i \cdots s_j = s_1 '  \cdots  s_r ' $ with $r \leqslant \min\{ m , j-i, n-j+i-2 \} $ for some $ 1\leqslant i < j \leqslant n $. We say that $G$ is $k$-chordal if we make $m=\infty $.
	\end{definition}
	
	By Remark \ref{r:reduced_relation}, it is easy to check that a group is $(k,m)$-chordal with respect to $S$ if and only if the Cayley graph $ \Gamma $ is $(k,m)$-chordal. 
	
	\begin{remark}\label{r:k1_chordal}A group $G$ is $(k,1)$-chordal if for every simple relation $ s_1 \cdots s_n = e $ in $S^{\pm 1}$ with $n \geqslant k$, then there exists $s \in S^{\pm 1}$ such that $ s_i \cdots s_j = s $ for some $ 1\leqslant i < j \leqslant n $  with $ j-i < n-2 $.
	\end{remark}
	
	\smallskip
	
	In \cite{M2} we proved the following:
	\begin{theorem}\label{T: chordal BP} If $\Gamma$ is a $(k,1)$-chordal graph, then $\Gamma$ is quasi-isometric to a tree.
	\end{theorem}
	
	By Theorems \ref{T: chordal BP} and \ref{T:charac} it is immediate the following:
	
	\begin{corollary}\label{c:k1_chordal}Let $G$ be a finitely generated group. If $G$ is $(k,1)$-chordal, then $G$ is virtually free.
	\end{corollary}
	
	\begin{definition}\label{d:ikm_chordal}Let $G$ be a group generated by $S$ and $i_0,k,m\in\NN $. We say that $G$ is $(i_0,k,m)$-chordal with respect to $S$ if for every simple relation $ s_1 \cdots s_n = e $ in $S^{\pm 1}$ with $n \geqslant k$, then there exist $s_1 ' , \ldots , s_r ' \in S^{\pm 1}$ such that $ s_i \cdots s_j = s_1 '  \cdots  s_r ' $ with $r \leqslant \min\{ m , j-i, n-j+i-2 \} $ for some $ 1\leqslant i \leqslant i_0 $ and $ i < j \leqslant n $. We say that $G$ is $[i_0 , k ]$-chordal if we make $m=\infty $.
	\end{definition}
	
	\begin{theorem}\label{t:ikm_chordal}Let $G$ be a group generated by $S$, $\Gamma = Cay(G,S)$ the Cayley graph and $k,m>0$. Then, $\Gamma$ is $\varepsilon$-densely $(k,m)$-chordal for some $\varepsilon > 0$ if and only if $G$ is $(i_0,k,m)$-chordal with respect to $S$ for some $i_0 > 0$.
	\end{theorem}
	
	\begin{proof}Let $G$ be a group generated by $S$ and suppose that the Cayley graph $\Gamma$ is $\varepsilon$-densely $(k,m)$-chordal for some $\varepsilon,k,m \in\NN $. Let $ s_1 \cdots s_n = e $ be a simple relation in $S^{\pm 1}$ with $n \geqslant k$ and consider $i_0=2\varepsilon$. Then, define the cycle $\gamma$ by $g_0 = e$, $g_l = s_1 \cdots s_l$ $\forall \, l\in\{1,\ldots, n \}$. Since $L(\gamma) = n \geqslant k$, then there exist strict shortcuts $\sigma _1 , \ldots , \sigma _t $ with $L(\sigma _i) \leqslant m \, \forall\, i, $ and such that their associated shortcut vertices define an $\varepsilon$-dense subset in $(\gamma,d_\gamma)$. Therefore, there exists a shortcut $\sigma _\lambda$ with $r = L(\sigma_\lambda) \leqslant m$ and associated shortcut vertices $ g_{i-1} ,g_j ,\, i < j,$ such that $ 1 \leqslant i \leqslant 2\varepsilon=i_0$. Since $\sigma _\lambda$ is a path from $g_{i-1}$ to $g_j$, then there exist $s_1 ' ,\ldots , s_r ' \in S^{\pm 1} $ such that $\sigma _\lambda$ is the path defined by $ h_0 = g_{i-1} ,\, h_l = g_{i-1} s_1 ' \cdots s_l ' $, $\forall \, l\in\{1,\ldots, r \}$ where $r \leqslant \min\{ j-i, n-j+i-2 \}$. Note that $$ g_{i-1} s_i \cdots s_j = g_j = h_r = g_{i-1} s_1 ' \cdots s_r ' $$ Thus, $ s_i \cdots s_j = s_1 ' \cdots s_r ' $ with $r \leqslant \min\{ m , j-i, n-j+i-2 \} $ with $ i \leqslant i_0 $ and $G$ is $(i_0,k,m)$-chordal.\\
		
		Reciprocally, suppose that $G$ is $(i_0,k,m)$-chordal with respect to $S$ for some $i_0,k,m\in\NN$. Let $\gamma$ be a cycle in $\Gamma$ with length $L(\gamma) = n \geqslant k$ and let $\varepsilon = i_0 $. 
		
		Fix a point $q\in \gamma$ and suppose $p$, $p'$ two adjacent vertices such that $q\in [pp']$ and $q\neq p'$. Let us label the ordered vertices in $\gamma$ as $g_0,\dots,g_n=g_0$ so that $p=g_0$ and $p'=g_1$. If we define $ s_l = g_{l-1} ^{-1} g_l,\, l\in\{ 1,\ldots,n \} $, then $ s_1 \cdots s_n = e $ is a simple relation. By hypothesis, there exist $i,j,r\in\NN$, $s_1 ' , \ldots , s_r ' \in S^{\pm 1}$ such that $ s_i \cdots s_j = s_1 '  \cdots  s_r ' $ with $r \leqslant \min\{ m , j-i, n-j+i-2 \} $ for some $ 1\leqslant i \leqslant i_0 $ and $ i < j \leqslant n $. Let $\sigma$ be the path defined by the vertices $ h_0 = g_{i-1} ,\, h_l = g_{i-1} s_1 ' \cdots s_l ' ,\, l\in\{ 1,\ldots , r \}$. Since $ h_r = g_{i-1} s_1 '  \cdots  s_r ' = g_{i-1} s_i \cdots s_j = g_j $ and $L(\sigma) = r \leqslant \min\{ j-i, n-j+i-2 \} = d_\gamma (g_{i-1} , g_j) $, then $\sigma$ contains an strict shortcut $\sigma _0$ with an associated shortcut vertex $g_{i-1}$. Since $q\in [g_0g_1], d_\gamma(q,g_{i-1})  < i_0 = \varepsilon  $ and $ L(\sigma_0) \leqslant L(\sigma) = r \leqslant m $.
	\end{proof}

	\begin{remark}\label{r:ikm_chordal}In fact, from the proof we obtain that if $\Gamma$ is $\varepsilon$-densely $(k,m)$-chordal, then $G$ is $(2\varepsilon ,k,m)$-chordal. On the other hand, if $G$ is $(i_0,k,m)$-chordal, then $\Gamma$ is $i_0$-densely $(k,m)$-chordal.
	\end{remark}
	
	\begin{remark}\label{r:ikm_chordal_infty}Theorem \ref{t:ikm_chordal} remains true for $ m'=m=\infty $.
	\end{remark}

	Hence, by Theorem \ref{t:qi_tree}, we get the following corollary.
	
	\begin{corollary}\label{c:qi_tree}A finitely generated group $G$ is quasi-isometric to a tree if and only if $G$ is $(i_0,k,m)$-chordal, for some $i_0,k,m$.
	\end{corollary}
	
	Thus, by Theorem \ref{T:charac}, we obtain the following result.
	
	\begin{corollary}\label{c:virtually_free}A finitely generated group $G$ is virtually free if and only if $G$ is $(i_0,k,m)$-chordal, for some $i_0,k,m$.
	\end{corollary}
	
	Let us recall also the following results.
	
	\begin{theorem}[Theorem 4, \cite{M1}]\label{t:hyperbolicity_1}If a graph $\Gamma$ is $\varepsilon$-densely $(k,m)$-chordal, then $\Gamma$ is $\delta$-hyperbolic with $ \delta \leqslant \max\{ \frac{k}{4} , \varepsilon + m \} $.
	\end{theorem}
	
	\begin{theorem}[Theorem 8, \cite{M1}]\label{t:hyperbolicity_2}If a graph $\Gamma$ is $\delta$-hyperbolic, then $\Gamma$ is $\varepsilon$-densely $k$-chordal, with $\varepsilon = 2\delta + 1$ and $k = 4\delta + 3$.
	\end{theorem}


	By Theorems \ref{t:ikm_chordal}, \ref{t:hyperbolicity_1} and \ref{t:hyperbolicity_2} we obtain also the following results.
	
	\begin{corollary}\label{c:ikm_chordal_hyperbolic_1}If a finitely generated group $G$ is $(i_0,k,m)$-chordal, then $G$ is $\delta$-hyperbolic, with $ \delta \leqslant \max\{ \frac{k}{4} , m + i_0 \} $.
	\end{corollary}
	
	\begin{corollary}\label{c:ikm_chordal_hyperbolic_2}If a finitely generated group $G$ is $\delta$-hyperbolic, then $G$ is $ [4\delta + 2 , 4\delta + 3] $-chordal.
	\end{corollary}
	
	Many groups may satisfy	being $k$-chordal. Let us see some examples.
	
	
	\begin{proposition}\label{p:Z_chordal} $\ZZ ^2 $ is 5-chordal.
	\end{proposition}
	
	\begin{proof} Let $\gamma$ be a cycle with length $L(\gamma) \geqslant 5$ and consider a simple relation $s_1 \cdots s_n = e$ associated with $\gamma$ in the sense of Remark \ref{r:reduced_relation}. Suppose without loss of generality that $s_1 = a$, then there is $1<i<n$ such that $s_1 = \ldots = s_{i-1} = a$ and $s_i = b^\varepsilon $, where $\varepsilon \in\{-1,1\}$. Now, consider $ i<j<n $ such that $ s_i = \ldots = s_{j-1} = b^\varepsilon $ and $s_j = a $ or $s_j = a^{-1} $. Since $\gamma$ is a cycle, we can assume, by relabelling if necessary, that $s_j = a^{-1}$. If we denote $q = \varepsilon (j-i) $, $ r = |q|$ and $s_1' = \ldots = s_r ' = b^\varepsilon$, then we get the following equality, see Figure \ref{Fig:shortcut}. $$ s_{i-1} s_i \cdots s_{j-1} s_j = a b^q a^{-1} = b^q = s_1 ' \cdots s_r ' $$ \begin{figure}[h]
			\centering
			\includegraphics[scale=1]{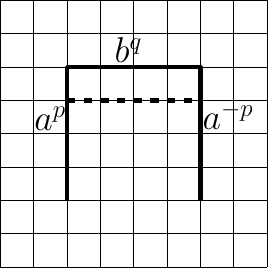}
			\caption{Shortcut $  a b^q a^{-1} = b^q  $.}
			\label{Fig:shortcut}
		\end{figure}
		
		Obviously $r < j-(i-1)$. If $s'_1\cdots s'_r$ defines a shortcut, we are done. Otherwise, if $r > n - j + (i-1) - 2$, then $\gamma$ is defined by $ s_1 = a,\, s_2 = \ldots = s_{r+1} = b^\varepsilon ,\, s_{r+2} = a^{-1},\, s_{r+3} = \ldots = s_n = b^{-\varepsilon} $, therefore $ s_1 \cdots s_n = ab^q a^{-1} b^{-q} = e$. If we denote $s_1 ' = a^{-1}$, then we get the following equality, see Figure \ref{Fig:shortcut_2}. $$ s_{r+1} s_{r+2}  s_{r+3} = b^\varepsilon a^{-1} b^{-\varepsilon} = a^{-1} = s_1 ' $$ \begin{figure}[h]
			\centering
			\includegraphics[scale=1]{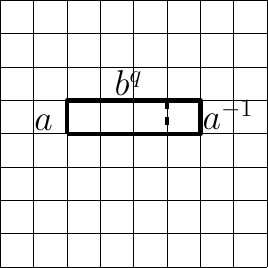}
			\caption{Shortcut $ b^\varepsilon a^{-1} b^{-\varepsilon} = a^{-1} $.}
			\label{Fig:shortcut_2}
		\end{figure}
		
		Since $L(\gamma) \geqslant 5$, then $|q|>1$ and therefore $ s_1 ' $ is a shortcut.  
	\end{proof}
	
	In the general case, since n-hypercubes are Hamiltonian graphs, we conjecture that $\ZZ ^n $ is $(2^n + 1)$-chordal, where $2^n$ is the number of vertices of an $n$-hypercube.
	
	\section{Baumslag-Solitar groups $BS(1,n)$}\label{BS}
	
	\begin{definition}For integers $m,n$, the Baumslag-Solitar group is defined by following presentation: $$ BS(m,n) = \left\langle a,b \,|\, ba^m b^{-1} = a^n \right\rangle  $$
	\end{definition}
	
	We will focus on $G_n := BS(1,n)$, $ n > 0 $. Let $ \Gamma_n$ denote the corresponding Cayley 2-complex built from rectangles homeomorphic to $ [0,1]\times [0,1] $, with both vertical sides labelled by $b$ upward, the top horizontal side labelled by an $a$ to the right, and the bottom horizontal side split into $n$ edges each labelled by $a$ to the right. See Figure \ref{Fig:leaf_bs}.
	
	\begin{figure}[h]
		\centering
		\includegraphics[scale=0.75]{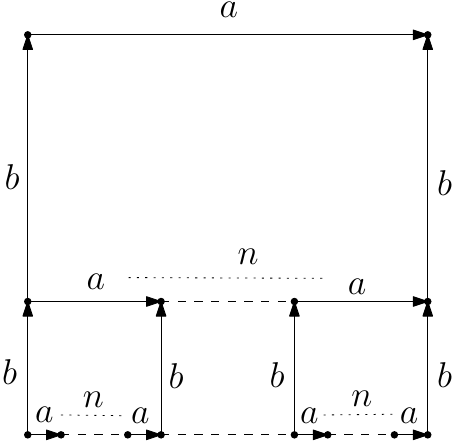}
		\caption{One part of the standard Cayley graph of $G_n$}
		\label{Fig:leaf_bs}
	\end{figure}
	
	Let $T_n$ be a Cantor tree where each vertex, besides having $n$ sons, has a father. Let us denote the horizontal projection as $\pi _n :\, \Gamma_n \longrightarrow T_n $, see Figure \ref{Fig:Baumslag-Solitar}. More details can be found in \cite{Eps}.
	
	\begin{figure}[h]
		\centering
		\includegraphics[scale=0.5]{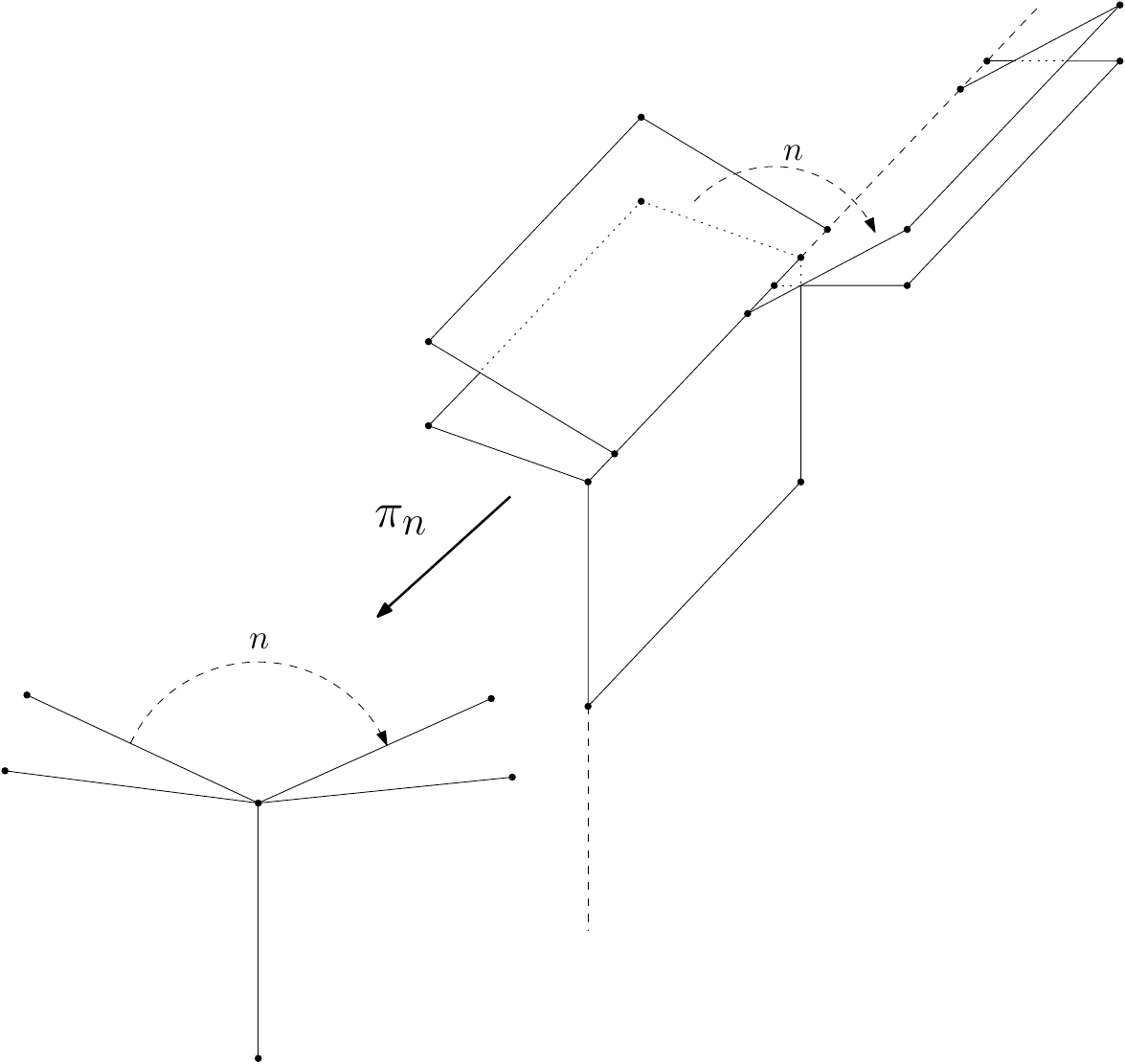}
		\caption{The Cayley graph of $G_n$ projected onto the Cantor tree $T_n$}
		\label{Fig:Baumslag-Solitar}
	\end{figure}

	\begin{proposition}\label{p:BS_eq}Given $s,t\in\NN$ and $\varepsilon,k,m > 0$ the following properties are equivalent.
		\begin{enumerate}
			\item $BS(s,t)$ is $\varepsilon$-densely $(k,m)$-chordal.
			\item $BS(-s,t)$ is $\varepsilon$-densely $(k,m)$-chordal.
			\item $BS(s,-t)$ is $\varepsilon$-densely $(k,m)$-chordal.
			\item $BS(-s,-t)$ is $\varepsilon$-densely $(k,m)$-chordal.
			\item $BS(t,s)$ is $\varepsilon$-densely $(k,m)$-chordal.
			\item $BS(-t,s)$ is $\varepsilon$-densely $(k,m)$-chordal.
			\item $BS(t,-s)$ is $\varepsilon$-densely $(k,m)$-chordal.
			\item $BS(-t,-s)$ is $\varepsilon$-densely $(k,m)$-chordal.
		\end{enumerate}
	\end{proposition}
	
	\begin{proof}The corresponding Cayley 2-complex of the groups $BS(s,t)$, $BS(-s,t)$, $BS(s,-t)$ and $BS(-s,-t)$ are the same except for the orientation of some labels on the edges. Therefore, these groups have Cayley graphs quasi-isometric, providing the proof of $ (1) \Leftrightarrow (2) \Leftrightarrow (3) \Leftrightarrow (4) $. 
		
		Since the corresponding Cayley 2-complex of $BS(t,s)$ is a $180^\circ $ rotation of the corresponding Cayley 2-complex of $BS(s,t)$, these groups have Cayley graphs quasi-isometric, providing the proof of $(1) \Leftrightarrow (5)$.
		
		The rest of equivalences are deduced by exchanging the role of the variables $s,t$.
	\end{proof}
	
	\smallskip
	
	Let us denote in $G_n$ as \textit{$a$-edges} the edges labelled by $a$ and \textit{$b$-edges} labelled by $b$.

	\begin{definition}Given $\omega = a^{\alpha _1} b^{\beta _1} \cdots a^{\alpha _k} b^{\beta _k} \in G_n$ we say that $h_n (\omega) = \sum _i \beta _i $ is the \textit{height} of $\omega$. Given $p$ any vertex in $T_n$ we denote $\widetilde{h}_n (p) = h_n (\omega)$ for some $\omega \in G_n$ such that $ \pi_n (\omega) = p $. Also, given a $b$-edge $e=\{v,w\} \in G_n$, let $h_n (e):=\max\{h_n(v),h_n(w)\} $.
	\end{definition}
	
	\begin{remark}If $\omega_1 , \omega_2 \in G_n$ are two points such that $\pi_n (\omega_1) = \pi_n (\omega_2) $, then $h_n (\omega_1) = h_n (\omega_2)  $, i.e., $\widetilde{h}_n (p)$ is well defined.
	\end{remark}


	
	Given a vertex $q\in T_n$, let us define: $$ L_q := \{ x \in \Gamma _n \,|\, d_{T_n}(\pi_n(x) , q) =  h_n(x) - \widetilde{h}_n (q) \}
	. $$ 
	Given $\omega \in L_q$, there is $\tilde{\omega} \in \pi_n ^{-1} (q) $ such that $ \tilde{\omega} b^{h(\omega) - h(\tilde{\omega})} = \omega $. We will consider $ p_q :\, L_q \longrightarrow \pi_n ^{-1} (q) $ the vertical proyection $ p_q (\omega) = \tilde{\omega} $ and we will consider an order in $\pi_n ^{-1} (q)$ where $x<y$ if $y=xa^\lambda$ for some $\lambda >0$. See Figure \ref{Fig:proyeccion_bs_2}.
	\begin{figure}[h]
		\centering
		\includegraphics[scale=0.75]{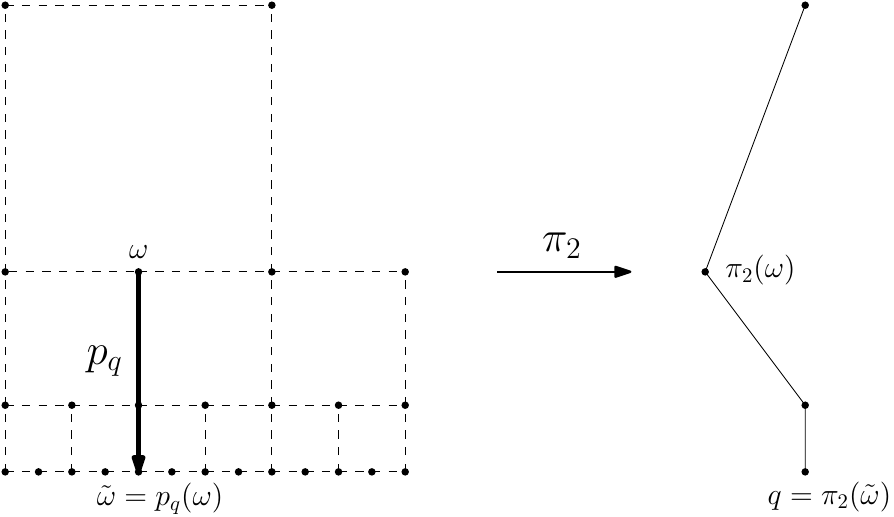}
		\caption{The projections $p_q$ and $\pi _n$.}
		\label{Fig:proyeccion_bs_2}
	\end{figure}
	
	\begin{definition}Given $x,y \in \Gamma _n$ we denote the geodesic path in $T_n$ from $ \pi _n (x) $ to $ \pi _n (y) $ as $ \widetilde{\sigma}_n (x,y) $. We say that the point $ q \in \widetilde{\sigma}_n (x,y) $ such that $ \widetilde{h}_n (q) = \min \widetilde{h}_n  (\widetilde{\sigma}_n (x,y))  $ is the branching point between $ x $ and $ y $. 
	\end{definition}
	
	\begin{remark}\label{r:branch}Let $x,y\in\Gamma _n$ and $\gamma$ be a path between $x$ and $y$. If $ q \in T_n $ is the branching point between $ x $ and $ y $, then $\gamma \cap \pi _n ^{-1} (q) \neq \emptyset$.
	\end{remark}
	
	\begin{definition}Let $ \gamma $ be a path in $\Gamma_n$ between two vertices $x,y \in V(\Gamma_n)$. Suppose that $ \gamma $ is defined by the vertices $x_0 = x$, $ x_i = x s_1 \cdots s_i $ for $i=1,\ldots , m$. We define the \textit{horizontal length} of $\gamma$ as follows: $$ H(\gamma) = \left| \{ i\in\{1,\ldots,m\} \,:\, \exists\varepsilon\in\{ -1,1 \},\, s_i = a^\varepsilon \} \right|.$$ 
	\end{definition}

	
	Notice that given two vertices in $G_n$, the horizontal length is not an invariant on the geodesics joining them as the following example shows. 
	
	\begin{example}In $G_3$, a geodesic defined by $a^3$ and a geodesic defined by $bab^{-1}$ conect the same pair of vertices but they do not have the same horizontal length. See Figure \ref{Fig:direct}.
	\end{example}
	
	\begin{figure}[h]
		\centering
		\includegraphics[scale=0.5]{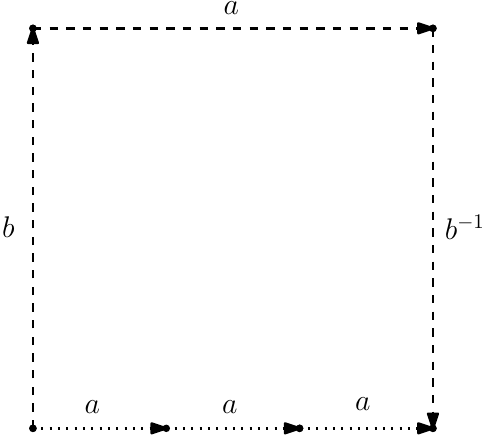}
		\caption{The geodesic given by $a^3$ has horizontal length 3 and the geodesic given by $bab^{-1}$ has horizontal length 1.}
		\label{Fig:direct}
	\end{figure}

	\begin{lemma}\label{l:heigth} If $\sigma$ is a geodesic in $G_2$, then $\sigma$ does not contain three $b$-edges $e_1,e_2,e_3$ with $h_n(e_1)=h_n(e_2)=h_n(e_3)$.
	\end{lemma}
	
	\begin{proof} Consider that the geodesic $\sigma$ is defined by the vertices $v_0,v_1,\dots ,v_n$ and let $s_i=v_{i-1}^{-1}v_i$ for every $1\leq i \leq n$. Suppose that there exist three $b$-edges $e_1=\{v_i,v_ {i+1} \},e_2=\{v_j,v_{j+1}\},e_3=\{v_k,v_{k+1} \}$ with $i<j<k$ and $h_n(e_i)=h_n(e_j)$ for every $1\leq i,j\leq 3$. Let us assume that these edges are consecutive with this property. Therefore, if  $v_{i+1}=v_ib$, then $v_{j+1}=v_jb^{-1}$ and $v_{k+1}=v_kb$, and if $v_{i+1}=v_ib^{-1}$, then $v_{j+1}=v_jb$ and $v_{k+1}=v_kb^{-1}$.
		
		\textbf{Case 1.} Suppose $v_{i+1}=v_ib$.  Then, it is immediate to check that 
		$$v_is_{j+2} \cdots s_k b s_{i+2}\cdots s_j=v_{k+1}$$ 
		(notice that it has the same $b$-edges and $a$-edges with the same orientations at the same heigths) and this defines a shortcut (two edges shorter) in $\sigma$ between $v_i$ and $v_{k+1}$ leading to contradiction since $\sigma$ is geodesic. See Figure \ref{Fig:LH1}.
		
		\begin{figure}[h]
			\centering
			\includegraphics[scale=0.75]{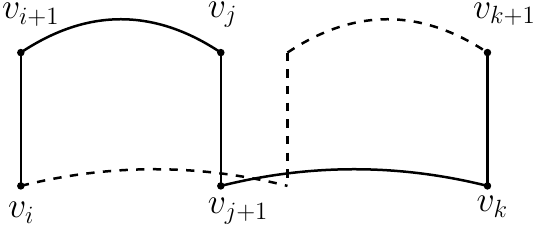}
			\caption{A shortcut in dashed line for Case 1.}
			\label{Fig:LH1}
		\end{figure}
		
		\textbf{Case 2.} Suppose $v_{i+1}=v_ib^{-1}$. Then, it is immediate to check that 
		$$v_is_{j+2} \cdots s_k b^{-1} s_{i+2}\cdots s_j=v_{k+1}$$ 
		and this defines a shortcut (two edges shorter) in $\sigma$ between $v_i$ and $v_{k+1}$ leading to contradiction. See Figure \ref{Fig:LH2}.
		
		\begin{figure}[h]
			\centering
			\includegraphics[scale=0.75]{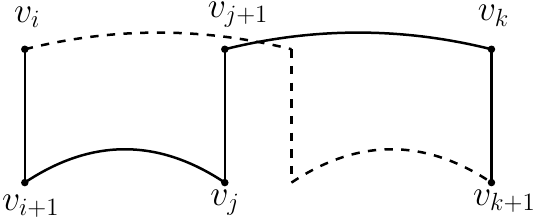}
			\caption{A shortcut in dashed line for Case 2.}
			\label{Fig:LH2}
		\end{figure}
	\end{proof}
	
	\begin{theorem}\label{t:BS_chordal}The group $G_2$ is 6-chordal with respect to $\{a,b\}$.
	\end{theorem}
	
	\begin{proof} Let $\gamma$ be a cycle in $\Gamma _2$ with length $L(\gamma) \geqslant 6$. 
		
		\textbf{Case 1.} Suppose that there are four consecutive vertices $x_i$ with $1\leq i\leq 4$ in $\gamma$ such that 
		$x_2=x_1b$, $x_3=x_2a$ and $x_4= x_3b^{-1}$. Then, $v_1 = x_1a, \, v_2 = x_1a^2=x_4$ defines a shortcut between $x_1$ and $x_4$ and we are done. 
		
		\textbf{Case 2.} Suppose, otherwise, that there are not such four consecutive vertices. Consider $M =\max_{v\in V(\gamma)}\{h_2 (v)\} $, $ m = \min_{v\in V(\gamma)}\{h_2 (v)\} $, $ q_m \in \pi _2 (\gamma) $ such that $ \widetilde{h}_2 (q_m) = m $ and the vertical projection $ p := p_{q_m}:\, L_{q_m} \longrightarrow \pi^{-1}(q_m) $. Let us denote $V_q = V(\gamma) \cap \pi_2 ^{-1} (q) $ for some $q\in T_2$. Then, since there exists at least one $a$-edge $e=\{v,w\}$ with $h_2(v)=h_2(w)=M$ and we are not in case 1, there exist three consecutive vertices in the cycle $ x_0,x,x_1 \in h_2^{-1}(M) $ with $ p(x_0) < p(x) < p(x_1) $, i.e., $x_0 = xa^{-1}$, $x_1 = xa$. Since $\gamma$ is a cycle, then $$ P = \{ z \in\gamma \,|\, p(x) = p(z) \text{ and } \pi _2 (x) \neq \pi _2 (z)  \} \neq \emptyset, $$ 
		where $z$ need not be a vertex.
		
		Let $ y \in P $ be such that $ d_\gamma (x,y) = \min _{v\in P} d_\gamma (x,v) $ and we will consider that $\gamma_y $ is the shortest path in the cycle $\gamma$ that joins $x$ and $y$. Fix $q_0 \in T_2$ the  branching point between $ x $ and $ y $ and $m_0 = \widetilde{h}_2  (q_0)$. 
		
		\textbf{Case 2a.} Suppose that $ y \in V(\gamma)$. Then $x,y$ are in the same plane and $h_2(y)=m_0$. Thus, the vertical path from $x$ to $ xb^{ - (M-m_0) }=y$ defines a shortcut between $x$ and $y$.
		
		\textbf{Case 2b.} Suppose that $ y \notin V(\gamma)$. We can suppose that $y$ is in an horizontal edge with vertices $y_0,y_1$, see Figure \ref{Fig:case2}. Suppose without loss of generality that $x_1, y_1 \in \gamma_y $.
		
		\begin{figure}[h]
			\centering
			\includegraphics[scale=0.75]{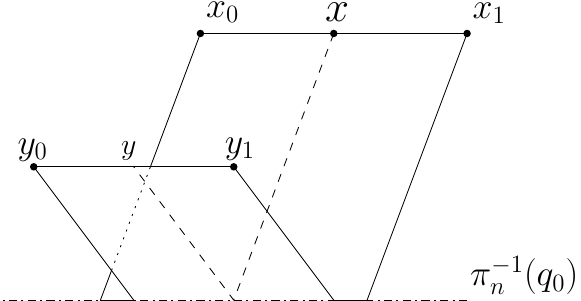}
			\caption{Case 2b.}
			\label{Fig:case2}
		\end{figure}
		
		If $\gamma_y$ is not a geodesic, since it is the shortest path in the cycle, then there is a shortcut between $x$ and $y$ and we are done. Suppose, otherwise that $\gamma_y$ is a geodesic. Then, by Lemma \ref{l:heigth}, there are not three $b$-edges in $\gamma_y$ with the same height.  
		
		Suppose $\gamma_y$ is defined by the vertices $x_1=v_0,v_1,\dots, v_n=y_1$ and let $s_i=v_{i-1}^{-1}v_i$ for every $1\leq i\leq n$.
		
		\textbf{Case 2b(i).} Suppose $h_2(y)=h_2(x)$. Then, since $x_1s_2\cdots s_n=y_1$ it is immediate to check that $xs_2\cdots s_n=y_0$ (it has the same $a$-edges and $b$-edges with the same orientations at the same heigths) and $d_\Gamma(x,y)\leq n-1+d_\Gamma(y_0,y)<n<L(\gamma_y)$ leading to contradiction since $\gamma_y$ is a geodesic.
		
		\textbf{Case 2b(ii).} Suppose $h_2(y)<h_2(x)$. Since there are not three $b$-edges in $\gamma_y$ with the same height, it is immediate to see that all vertices in $\gamma_y\cap \pi^{-1}(q_0)$ must be consecutive. Also, since $b^{-1}a^kb=ab^{-1}a^{k-2}b$ and $b^{-1}a^{-k}b=a^{-1}b^{-1}a^{-k+2}b$ for every $2\leqslant k\in \ZZ$ it is readily seen that $\gamma_y\cap \pi^{-1}(q_0)$ is a single edge, $\{v_i,v_{i+1}\}$ for some $0<i<n$ (otherwise, these relations allow a shortcut in $\gamma_y$ leading to contradiction).
		
		Since $h_2(x_1)=M=\max_{v\in V(\gamma)}\{h_2(v)\}$ and there are not three $b$-edges in $\gamma_y$ with the same height it follows that all $b$-edges $\{v_j,v_{j+1}\}$ with $ j <i$ satisfy that $v_{j+1}=v_jb^{-1}$.
		
		Also, let us see that all $b$-edges $\{v_j,v_{j+1}\}$ with $ j >i$ satisfy that $v_{j+1}=v_jb$. Suppose otherwise that $v_{j+1}=v_jb^{-1}$ for some $j>i$. If $h_2(v_j)\leq h_2(y_1)$ then necessarily there must be three $b$-edges with the same heigth leading to contradiction with Lemma \ref{l:heigth}. If $h_2(y_1)<h_2(v_j)\leq h_2(x_1)$, then there are at least three $b$-edges with heigth $h_2(y_1)+1$, one between $x_1=v_0$ and $v_i$, one between $v_i$ and $v_j$ and one between $v_{j}$ and $v_n=y_1$, leading to contradiction.
		
		Therefore, $L(\gamma_y)= H(\gamma_y)+h_2(x_1)+h_2(y_1)-2m_0$.
		
		Let us build a new geodesic $\sigma$ from $x$ to $y$ as follows. For each $a$-edge $\{v_j,v_{j+1}\}$ with $ j >i$ and $v_{j+1}=v_ja^{\varepsilon}$ with $\varepsilon\in \{-1,1\}$, consider any vertex $v_k$ with $k<i$ such that $h_2(v_k)=h_2(v_j)$. Then it is readily seen that the path $xs_1\cdots s_{k}a^{\varepsilon}s_{k+1}\cdots s_{j}s_{j+2}\cdots s_n$ is also a path from $x$ to $y_1$ with the same length as the restriction of $\gamma_y$ joining $x$ to $y_1$. Therefore, repeating this change for every $a$-edge betweeen the lowest $a$-edge and $y_1$, we build a path $\sigma$ from $x$ to $y$ such that $L(\sigma)=L(\gamma_y)$ and such that if $V(\sigma) \cap \pi_2 ^{-1} (q_m)=\{w_r,w_{r+1}\}$, $w_{s+1}=w_sb$ for every $s>r$, this is, $\sigma$ is a path that after its lowest $a$-edge is just a vertical geodesic. Since $\gamma_y$ is a geodesic, $\sigma$ is also a geodesic from $x$ to $y$ containing $x_1$ and $y_1$. 
		
		Let $l=h_2(x)-h_2(y)$. Then, let $z=xb^{-l}a$. By construction, notice that $p(y_1)<p(z)<p(x_1)$. If there is a vertex $w\neq x$ in $\sigma$ such that $p(w)=p(x)$ then the vertical geodesic from $x$ to $w$ is a shortcut in $\sigma$ leading to contradiction. Otherwise, there is a vertex $w$ in $\sigma$ such that $p(w)=p(z)$, see Figure \ref{Fig:claim1}.
		
		\begin{figure}[h]
			\centering
			\includegraphics[scale=0.5]{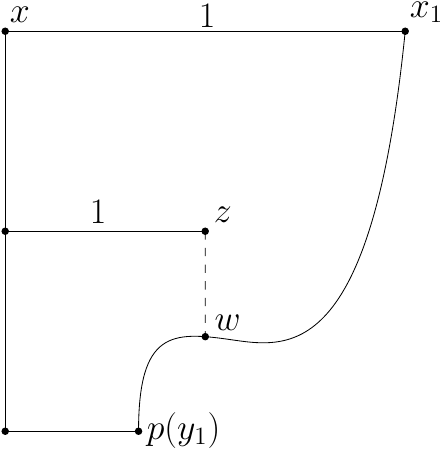}
			\caption{If there is no vertex $w\neq x$ in $\sigma$ with $p(w)=p(x)$, then there is a vertex $w$ in $\sigma$ with $p(w)=p(z)$.}
			\label{Fig:claim1}
		\end{figure}
		
		However, the path defined by $x, xb^{-1},\dots,xb^{-l},xb^{-l}a=z$ followed by the vertical geodesic from $z$ to $w$ 
		has length $h_2(x)-h_2(w)+1$ and is strictly shorter than the restriction of $\sigma$. Thus, this defines a shortcut 
		in $\sigma$ leading to contradiction.	
	\end{proof}

	By Proposition \ref{p:BS_eq} and Theorem \ref{t:BS_chordal}, we get the following result.
	
	\begin{corollary}The groups $BS(1,2)$, $BS(-1,2)$, $BS(1,-2)$, $BS(-1,-2)$, $BS(2,1)$, $BS(-2,1)$, $BS(-2,-1)$ and $BS(-2,-1)$ are 6-chordal with respect to $\{a,b\}$.
	\end{corollary}
	
	\begin{lemma}\label{l:geodesic}Let $ n>2 $ and $ x,y \in V(\Gamma _n) $. If there exists $ k,k'\in\NN $ such that $ \tilde{x} = xb^k $, $ \tilde{y} = yb^{k'} $ and $ \tilde{y} = \tilde{x}a $, then there is a geodesic $ \sigma $ between $x$ and $y$ such that $ \tilde{x}, \tilde{y} \in \sigma $.
	\end{lemma}
	
	\begin{proof}Let us denote $\sigma$ the path between $x$ and $y$ defined by $b^k a b^{-k'}$. Then $ L(\sigma) = k + k' + 1 $. Let $\sigma '$ be other path between $x$ and $y$. Denote $ h = \max\{ h_n (x) : x\in\sigma \} $ and $ h' = \max\{ h_n (x') : x'\in\sigma' \} $. 
		
		If $ h' \geqslant h $, then $ L(\sigma ') \geqslant h'- h_n (x) + h'- h_n (y) + H(\sigma ') \geqslant k + k' + 1 = L(\sigma) $. 
		
		Suppose otherwise $h' < h$, then $ n^{h-h'} + 2h' - h_n (x) - h_n(y) \leqslant L(\sigma ')  $. Since $n>2$, then $ 2x + 1 \leqslant n^x $ for all $x\geqslant 1$. Therefore $ 2(h-h') + 1 \leqslant n^{h-h'} $. \[ \begin{array}{c}
			2h + 1 \leqslant n^{h-h'} + 2h' \\ 
			2h -h_n (x) - h_n (y) + 1 \leqslant n^{h-h'} + 2h' -h_n (x) - h_n (y)\\ 
			L(\sigma) \leqslant n^{h-h'} + 2h' -h_n (x) - h_n (y) \leqslant L(\sigma ')\\ 
			L(\sigma) \leqslant L(\sigma ')
		\end{array}  \]
	\end{proof}
	
	\begin{theorem}\label{t:BS_no_chordal}$G_n$ is not $k$-chordal with respect to $\{a,b\}$ for every $n > 2$, $k>0$.
	\end{theorem}
	
	\begin{proof}Fix $N > 1$. Let us define the cycle $\gamma _N$ by its vertices as follow. $ x_0 = e $, $ x_1 = a $ , $x_2 = ab$, ..., $x_{N+1} = ab^N$, $x_{N+2} = ab^N a$, $x_{N+3} = ab^N a b^{-1}$, ..., $x_{2N+2} = ab^N a b^{-N}$, $x_{2N+3} = ab^N a b^{-N} a^{-1}$, $x_{2N+4} = ab^N a b^{-N} a^{-1} b$, ..., $x_{3N+3} = ab^N a b^{-N} a^{-1} b^N$, $x_{3N+4} = ab^N a b^{-N} a^{-1} b^N a^{-1}$, $x_{3N+5} = ab^N a b^{-N} a^{-1} b^N a^{-1} b^{-1}$, ..., $x_{4N+4} = ab^N a b^{-N} a^{-1} b^N a^{-1} b^{-N} = x_0 $. See Figure \ref{Fig:counterexample}.
		
		\begin{figure}[h]
			\centering
			\includegraphics[scale=1]{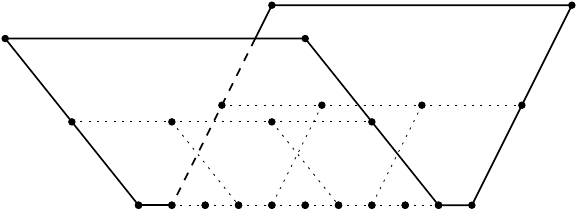}
			\caption{Example of $\gamma _2$ in $G_3$}
			\label{Fig:counterexample}
		\end{figure}
		
		Claim. $ \gamma _N $ has not shortcuts.
		
		Let $x_i,x_j\in V(\gamma)$. Denote $ L_1 = \{ 1 , ... , N+1 \} $, $ L_2 = \{ N+2, ... , 2N+2 \} $, $L_3 = \{ 2N + 3 , ... , 3N + 3 \}$, $ L_4 = \{ 3N + 4 , ... , 4N + 4 \} $, $O = \pi _n (e)$.
		
		Case 1. Suppose $ i,j\in L_m $ for $m\in\{1,2,3,4\}$. Since $ p_O(x_i) = p_O(x_j)  $, then $ \gamma _N $ defines a vertical geodesic between $x_i$ and $x_j$. Therefore, there is not a shortcut between $x_i$ and $x_j$.
		
		Case 2. Suppose $ i\in L_1 $, $ j\in L_2 $ (or $ i\in L_3 $, $ j\in L_4 $). By Lemma \ref{l:geodesic}, $ \gamma _N $ defines a geodesic between $ x_i $ and $x_j$ and therefore there is not a shortcut between $x_i$ and $x_j$.
		
		Case 3. Suppose $ i\in L_1 $, $ j\in L_3 $ (or $ i\in L_2 $, $ j\in L_4 $). Let $\sigma$ be a path between $x_i$ and $x_j$. Suppose without loss of generality that $h_n(x_i) \geqslant h_n(x_j)$. Denote $ h = \max\{ h_n (x) : x\in\sigma \} $. 
		
		If $ h \geqslant N $, then $L(\sigma) \geqslant h_n(x_j) + h + (h-h_n(x_i)) + H(\sigma) \geqslant 2h + 2 + h_n (x_j) - h_n(x_i) \geqslant 2N + 2 + h_n (x_j) - h_n(x_i) = d_{ \gamma_N } (x_i , x_j) $. 
		
		Otherwise, let us assume $h < N$. On one side $ n^{N-h} + 1 \leqslant H (\sigma) $, on the other hand the number of edges in $\sigma$ labelled by $b$ is at least $ H = h_n (x_i) + h_n(x_j) + 2 (h - h_n (x_i) ) $, therefore $ n^{N-h} + 1 + H  \leqslant L(\sigma)  $. Since $n>2$, then $ 2x + 1 \leqslant n^x $ for all $x\geqslant 1$. Therefore $ 2(N-h) + 1 \leqslant n^{N-h} $. 
		
		\[ \begin{array}{c}
			2(N-h) + 2 + H  \leqslant n^{N-h} + 1 + H \\ 
			d_{ \gamma_N } (x_i , x_j) = 2(N-h) + 2 + H \leqslant n^{N-h} + 1 + H  \leqslant L(\sigma )\\ 
			d_{ \gamma_N } (x_i , x_j) \leqslant L(\sigma )
		\end{array}  \]
		
		Thus, there is not a shortcut between $x_i$ and $x_j$.
		
		Case 4. Suppose $i\in L_1, j\in L_4$ (or $i\in L_2, j\in L_3$). In this case, if $\sigma$ is a path between $x_i$ and $x_j$, then $H(\sigma) > 0$ and $ L(\sigma) \geqslant h_n (x_i) + h_n(x_j) + H(\sigma) \geqslant h_n (x_i) + h_n(x_j) + 1 = d_{ \gamma_N } (x_i , x_j) $
		
		Therefore, $ \sigma $ is not a shortcut, and the proof of the claim is finished.
		
		Suppose that $G_n$ is $k$-chordal for some $k>0$. Then $\gamma_k$ has a shortcut, which contradicts the claim.
		
	\end{proof}
	
	\section{Word problem}\label{word}

	Let $S$ be a finite set, $F(S)$ the free group generated by $S$ and $R\subseteq F(S)$ a finite set. We consider a finite group presentation $ G = \left\langle S | R \right\rangle $. We will denote $a=_H b$ the equality $a=b$ in the group $H$, for example, two different relators $r_1,r_2 \in R$ are equal in $G$, i.e., $r_1 =_G r_2$, but they are not equal in $ F(S) $, i.e., $ a \neq _{F(S)} b$. Let $\omega\in F(S)$ be such that $ \omega =_G e $. The \emph{area} of $\omega$ is defined as follows: $$ A(\omega) = \min\{ m \,|\, \omega =_{F(S)} u_1 r_1 ^{\varepsilon _1} u_1 ^{-1} \cdots u_m r_m ^{\varepsilon _m} u_m ^{-1} ,\, u_i \in F(S) ,\, r_i \in R,\, \varepsilon\in\{ -1,1 \} \} $$
	
	The \emph{Dehn function} of the presentation $ \left\langle S | R \right\rangle $ is defined as follow: $$ \text{Dehn}(n) = \max\{ A(\omega) \,|\, \omega \in \left\langle \left\langle R \right\rangle \right\rangle  ,\, |\omega| \leqslant n \}, $$
	where $\left\langle \left\langle R \right\rangle \right\rangle=\{\prod_{i=1}^M u_ir_i^{\varepsilon_i} u_i^{-1} \, | \,  M\in \mathbb{N}, \, u_i\in F(S), \, r_i\in R, \,  \varepsilon\in\{ -1,1 \}\}$. See \cite{Sh}.
	
	A function $f:\,\NN \longrightarrow \RR$ is an \emph{isoperimetric function} for the presentation $ \left\langle S | R \right\rangle $ if it is a non-decreasing function such that $ \text{Dehn} \leqslant f $. 
	
	\begin{theorem}[Theorem 1.1, \cite{BRS}]\label{t:iso_recursive} A finite presentation has a solvable word problem if and only if it admits a recursive isoperimetric function.
	\end{theorem}
	
	\begin{theorem}\label{t:word_problem}Let $ G = \left\langle S | R \right\rangle $ be a finite group presentation. If $G$ is $k$-chordal, then $\left\langle S | R \right\rangle$ admits a recursive isoperimetric function, i.e., $\left\langle S | R \right\rangle$ has a solvable word problem.
	\end{theorem}
	
	\begin{proof}Suppose that $G$ is k-chordal with respect to $S$. Since $\left\langle S | R \right\rangle$ is a finite presentation, then the set of words with length less than or equal to $k$ is finite, and we can compute $c = Dehn(k) $. We define the following recursive function: 
		
		$$ \left\lbrace \begin{array}{ll}
			f(n) = c & \text{if }  n = 1,\ldots , k \\ 
			f(n) = 2 f(n-1) = c 2^{n-k} & \text{if } n > k
		\end{array}  \right.  $$\\
		
		We claim that $ \text{Dehn}(n) \leqslant 2 \text{Dehn}(n-1) $ if $ n>k $. In fact, let $ \omega \in \left\langle \left\langle R\right\rangle \right\rangle  $ such that $ \text{Dehn}(n) = A(\omega) $ and suppose that $ |\omega| = n > k $. 
		
		Suppose $ \omega = s_1 \cdots s_n = e $ is a simple relation for some $s_1,\ldots , s_n \in S^{\pm 1}$. Since $G$ is $k$-chordal, then there exist $s_1 ' , \ldots , s_r ' \in S^{\pm 1}$ be such that $ s_i \cdots s_j = s_1 '  \cdots  s_r ' $ with $r \leqslant \min\{ j-i, n-j+i-2 \} $ for some $ 1\leqslant i < j \leqslant n $. We can split $\bar{\omega} = \omega _1 \omega _2$ where $ \omega _1 = s_{i+1} \cdots s_j s_r'^{-1} \cdots s_1'^{-1} $, $ \omega _2 = s_1' \cdots s_r' s_{j+1} \cdots s_n s_1 \cdots s_i $ and $\bar{\omega} $ is a cyclic conjugate of $\omega$. Suppose that $D_1$ and $D_2$ define a van Kampen diagram for $\omega_1$ and $\omega_2$ over $G$ (with base point $ g = s_1\cdots s_i$) respectively, see Figure \ref{Fig:D1_D2}.
		
		\begin{figure}[h]
			\centering
			\includegraphics[scale=0.5]{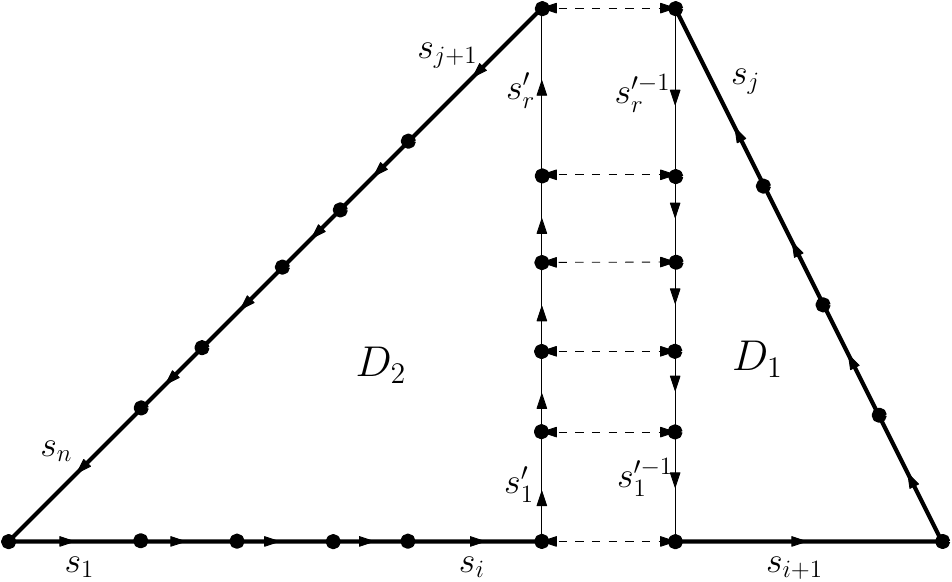}
			\caption{Van Kampen diagram for $\omega_1$ at right and Van Kampen diagram for $\omega_2$ at left.}
			\label{Fig:D1_D2}
		\end{figure}
		
		If we join $D_1$ and $D_2$ through the common vertices $s_1 ' , \ldots , s_r ' $ (see Figure \ref{Fig:D}), we obtain a van Kampen diagram $D$ for $\bar{\omega} $ over $G$ (with base point $ g $).
		
		\begin{figure}[h]
			\centering
			\includegraphics[scale=0.5]{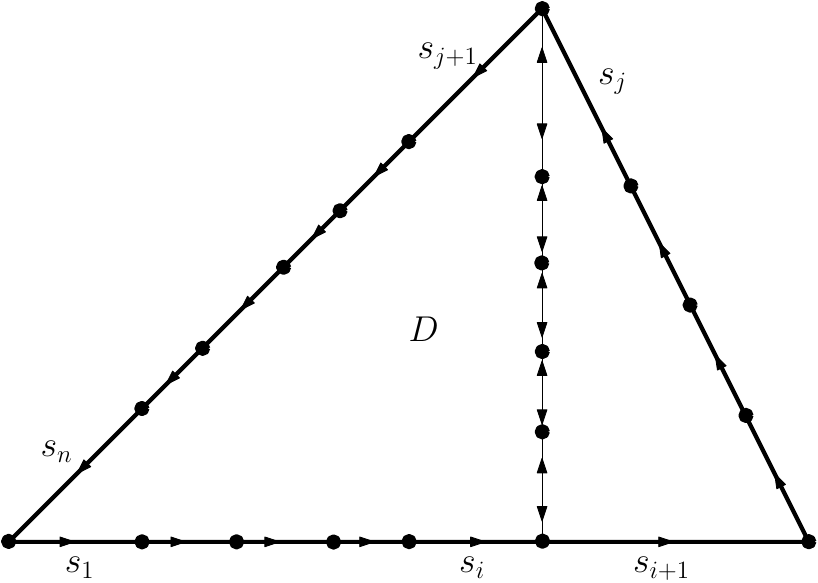}
			\caption{Van Kampen diagram for $\omega$.}
			\label{Fig:D}
		\end{figure}
		
		If $ \omega = e $ is not a simple relation, then we can split $\bar{\omega} = \omega _1 \omega _2$ where $ \omega _1 = e$, $ \omega _2 = e $ and $\bar{\omega} $ is a rotation of $\omega$.
		
		Thus $ \text{Dehn}(n) = A(\omega) = A(\bar{\omega}) \leqslant A(\omega_1) + A(\omega_2) \leqslant \text{Dehn}(n-1) + \text{Dehn}(n-1) = 2\text{Dehn}(n-1) $.
		
		By induction, $ \text{Dehn}(n) \leqslant f(n) $ for all $n\in\NN$. We have shown that $ \left\langle S | R \right\rangle $ admits a recursive isoperimetric function, and by Theorem \ref{t:iso_recursive}, the word problem is solvable.
	\end{proof}

	It is well known that finitely generated free abelian groups have solvable word problem. For the particular case of the group $\ZZ ^2$, since it is 5-chordal (Proposition \ref{p:Z_chordal}), it follows also from Theorem \ref{t:word_problem} that $\ZZ ^2$ has solvable word problem.
	
	We have seen in Section \ref{BS} that $G_2$ is 6-chordal (Theorem \ref{t:BS_chordal}). Therefore, by Theorem \ref{t:word_problem}: 
	
	\begin{corollary} The Baumslag-Solitar group $ G_2$ has solvable word problem.
	\end{corollary}

	Given two functions $f,g:\, \NN \longrightarrow \RR$, we will denote $ f \preccurlyeq g $ if there exists $C>0$ such that $f(n) \leqslant Cg(Cn+C) + Cn + C$ for all $n\in\NN$. This gives an equivalence relation $f \approx g \Longleftrightarrow f \preccurlyeq g $ and $ g \preccurlyeq f $. We denote the exponential function in base $m$ as $ \exp _m (n) = m^n $. This simplifies the notation in some cases, for example, $ exp_m ^k (n) = m ^{m^{\dots ^{m^n}}} $.
	
	\begin{remark}\label{r:cota}Let $ G = \left\langle S | R \right\rangle $ be a finite group presentation. If $G$ is $k$-chordal and $c = Dehn(k) $, then the Dehn function is bounded by the following function:
		$$ \left\lbrace \begin{array}{ll}
			f(n) = c & \text{if }  n = 1,\ldots , k \\ 
			f(n) = 2 f(n-1) = c 2^{n-k} & \text{if } n > k
		\end{array}  \right.  $$
		Thus, if $ G $ is $k$-chordal, then $Dehn \preccurlyeq \exp _2 $.
	\end{remark}
	
	This idea provides us with a method to detect groups with solvable word problem that are not  $k$-chordal for any $k\in\NN$.
	
	\begin{example}\label{e:hydra}We will consider the group $$\Gamma _k = \left\langle a_1 , \ldots , a_k , t , p \, \mid \, t^{-1} a_i t = a_i a_{i-1}  ,\, [t,a_1] = 1 ,\, [p,a_j t] = 1, \, i>1,\, j>0 \right\rangle $$ In \cite{DR}, the authors prove that $ Dehn \approx A_k $, where $ A_3 (n) = \exp _2 ^n (1) $ and $ A_{m+1} \geqslant A_m $. By Theorem \ref{t:iso_recursive}, $\Gamma _k$ has a solvable word problem. However, if $\Gamma _k$ were $l$-chordal, by  Remark \ref{r:cota}, then $Dehn \preccurlyeq \exp _2 $, which is a contradiction for $k\geqslant 3$. Thus, we have shown that $ \Gamma _k $ is not chordal for $k\geqslant 3$.
	\end{example}

	\section{Acknowledgments}
	The authors are very grateful to Yago Antol\'in for his helpful comments and suggestions.

\end{document}